\documentclass[letterpaper,11pt,twoside,reqno]{amsart}
\usepackage[pagebackref,colorlinks=true,urlcolor=blue,linkcolor=blue,citecolor=blue]{hyperref}
\usepackage[top=1.65in,bottom=1.7in,left=1.5in,right=1.5in]{geometry}
\usepackage{times}
\usepackage{amsmath}
\usepackage{amssymb}
\usepackage{amsfonts}
\usepackage{amscd}
\usepackage{amsthm}
\usepackage{amsxtra}

\usepackage[all]{xy}  \CompileMatrices
\usepackage{mathrsfs}
\usepackage{nicefrac}
\usepackage{graphicx}
\usepackage{color}
\usepackage{enumerate}
\usepackage{url}

\usepackage{algorithm}

\theoremstyle{plain}
\newtheorem{theorem}{Theorem}

\newtheorem*{proposition*}{Proposition}

\newtheorem*{corollary*}{Corollary}
\newtheorem{lemma}[theorem]{Lemma}

\newtheorem*{theorem*}{Theorem}
\newtheorem*{lemma*}{Lemma}
\newtheorem*{conjecture*}{Conjecture}
\newtheorem{conjecture}[theorem]{Conjecture}
\newtheorem*{question*}{Question}

\newtheorem*{problem*}{Problem}

\theoremstyle{definition}
\newtheorem{definition}[theorem]{Definition}
\newtheorem{example}[theorem]{Example}

\newtheorem*{exercise*}{Exercise}

\theoremstyle{remark}
\newtheorem{remark}[theorem]{Remark}
\newtheorem*{remark*}{Remark}
\newtheorem{remsTh}[theorem]{Remarks}

\newcommand{\subclass}[1]{}

\newcommand{\enumTi}[1]{\renewcommand{\theenumi}{#1}}

\newcommand{\alphenumi}{\enumTi{\alph{enumi}}}

\newcommand{\romenumi}{\enumTi{\roman{enumi}}}

\newcommand{\lt}{\left}
\newcommand{\rt}{\right}

\newcommand{\sabs}[1]{{\lvert{#1}\rvert}}

\newcommand{\widebar}[1]{\overline{#1}}

\newcommand{\nfrac}[2]{{\nicefrac{#1}{#2}}}

\newcommand{\RR}{\mathbb{R}}

\newcommand{\eps}{\varepsilon}

\newlength{\algotabbingwidth}
\renewcommand{\paragraph}[1]{\par\vspace{1ex}\noindent #1}

\numberwithin{theorem}{section}

\newcommand{\myparagraph}[1]{\paragraph{#1}}

\begin{document}
\setcounter{tocdepth}{3}

%


\title[Good edge-labelings]{Good edge-labelings and graphs with girth at least five}
%
%
\author{Michel Bode}%

\author{Babak Farzad}%
\address{Babak Farzad: Mathematics Department\\ Brock University\\St.~Catharines, Canada}
\email{bfarzad@brocku.ca}
\thanks{This research has been partially funded by an NSERC Discovery Grant.  Part of the work on this project was done while DOT was a Visiting International
  Scholar at Brock University.}

\author{Dirk Oliver Theis}
\address{Michel Bode \& Dirk Oliver Theis:
  Fakult\"at f\"ur Mathematik\\
  Otto-von-Guericke-Universit\"at Magdeburg\\
  Universit\"atsplatz~2\\
  39106~Magdeburg\\
  Germany}%
\email{
theis@ovgu.de {\tiny\href{http://dirkolivertheis.wordpress.com}{http://dirkolivertheis.wordpress.com}}}%
\thanks{For funding his visit to CanaDAM 2011, MB thanks the Faculty of Mathematics of the Otto von Guericke Univeristy Magdeburg and CanaDAM student funding.}%


\subjclass[2000]{Primary 05C78}





\begin{abstract}
  A good edge-labeling of a graph [Ara{\'u}jo, Cohen, Giroire, Havet, Discrete Appl.~Math., forthcoming] is an assignment of numbers to the edges such that
  for no pair of vertices, there exist two non-decreasing paths.
  In this paper, we study edge-labeling on graphs with girth at least~5.  In particular we verify, under this additional hypothesis, a conjecture by Ara{\'u}jo
  et al.  This conjecture states that if the average degree of $G$ is less than~$3$ and $G$ is minimal without an edge-labeling, then $G \in \{C_3,K_{2,3}\}$.
  (For the case when the girth is~4, we give a counterexample.)
\end{abstract}


\date{Fri Jul 27 14:10:27 CEST 2012}
\maketitle





\section{Introduction}

All graphs are finite and simple.  We refer to Diestel~\cite{Diestel-GTIII} for most of our graph theory terminology.

A \textit{good edge-labeling}~\cite{BermondCosnardPerennes09} of a graph $G$ is a labeling of its edges $\phi\colon E(G) \to \RR$ such that, for any ordered pair of
vertices $u$ and $v$, there is at most one nondecreasing path from $u$ to~$v$.  We will mostly use the following characterization of a good edge-labeling, which
involves cycles instead of pairs of paths:
\begin{equation*}
  \text{\it An edge-labeling is good, if, and only if, every cycle has at least two local minima.}
\end{equation*}
Here, by a local minimum we mean an edge~$e$ whose label is strictly less than the labels of the two edges incident to~$e$ on the cycle (this differs from the
definition in the next section because at this point, unlike later in the paper, for convenience, we assume that all labels are distinct).

Good edge-labelings have first been studied by Bermond, Cosnard, and P\'erennes~\cite{BermondCosnardPerennes09} in the context of so-called Wavelength Division
Multiplexing problems~\cite{BermondCosnardCoudertPerennes06}.  There, given a network, the so-called Routing and Wavelength Assignment Problem asks for finding
routes and associated wavelengths, such that a set of traffic requests is satisfied, while minimizing the number of used wavelengths.

Araujo, Cohen, Giroire, and Havet~\cite{AraujoCohenGiroireHavet09,AraujoCohenGiroireHavet11} have studied good edge-labelings in more depth.  They call a graph with
no good edge-labeling \textit{bad}, and say that a \textit{critical} graph is a minimal bad graph, that is, every proper subgraph has a good edge-labeling.  It is
easy to see that $C_3$ and $K_{2,3}$ are critical.  Araujo et al.'s~\cite{AraujoCohenGiroireHavet11} paper comprises an infinite family of critical graphs; results
that graphs in some classes always have a good edge-labelings (planar graphs with girth at least 6, $(C_3,K_{2,3})$-free outerplanar graphs, $(C_3,K_{2,3})$-free
sub-cubic graphs); the algorithmic complexity of recognizing bad graphs; and a connection to matching-cuts.  (A \textit{matching-cut}, aka ``simple
cut''~\cite{Graham70}, is a set of independent edges which is an edge-cut.)

In fact, all their arguments for proving non-criticality rely on the existence of matching-cuts.  One of the central contributions of our paper is that we move
beyond using matching-cuts.

Araujo et al.\ also pose a number of problems and conjectures.  In particular, they have the following conjecture, which is one of the two motivations behind our
paper.
\begin{conjecture}[Araujo et al.~\cite{AraujoCohenGiroireHavet11}]\label{conj:main}
  There is no critical graph with average degree less than~3, with the exception of~$C_3$ and~$K_{2,3}$.
\end{conjecture}
Araujo et al.~\cite{AraujoCohenGiroireHavet11} prove a weaker version of this conjecture.  They establish the existence of a matching-cut, relying in part on a
theorem by Farley \& Proskurowski \cite{FarleyProskurowski84,BonsmaFarleyProskurowski2011} stating that a graph with sufficiently few edges always has a matching
cut.  They also use a characterization of extremal graphs with no matching-cut by Bonsma~\cite{Bonsma05,BonsmaFarleyProskurowski2011}.  From the proofs in Araujo et
al.~\cite{AraujoCohenGiroireHavet11}, it appears that the depths of the arguments increases rapidly as the upper bound~$3$ is approached.

In this paper, we show that there is no critical graph with average degree less than three and girth at least five.  Put differently, we prove
Conjecture~\ref{conj:main} in the case when the graph has girth at least five.

\begin{theorem}\label{thm:main}
  There is no critical graph with average degree less than three and girth at least five.
\end{theorem}

Moreover, we falsify Conjecture~\ref{conj:main} for the case of girth~4: Fig.~\ref{fig:girth-4-counterexp} shows a graph with girth~4 and average degree
$\frac{26}{9} < 3$ (9 vertices, 13 edges), which does not contain either $C_3$ or $K_{2,3}$ as a subgraph.  We leave to the reader as an exercise to argue that the
graph has no good edge labeling.  It can easily be verified that every proper subgraph has a good edge labeling, so the shown graph is critical.  In other words,
Fig.~\ref{fig:girth-4-counterexp} shows a counterexample to Conjecture~\ref{conj:main} for the case of girth~4.

\begin{figure}[htp]
  \centering
  \scalebox{.5}{\input{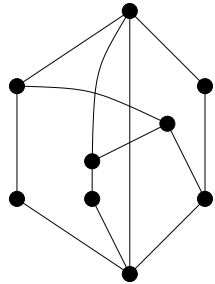}}
  \caption{Critical graph with girth~4 and average degree $< 3$}\label{fig:girth-4-counterexp}
\end{figure}

Another motivation behind our paper is to demonstrate how large girth makes labeling arguments easier.\footnote{%
  Indeed, until very recently, no bad graph with girth larger than four was known.  In particular, the bad graphs in Araujo et al.'s construction contain many
  4-cycles.  This fact had led us to conjecture, that there exists a finite number~$g$ such that every graph with girth at least~$g$ has a good edge-labeling; as
  mentioned above, Araujo et al.~\cite{AraujoCohenGiroireHavet11} have shown that with the additional restriction that the graphs be planar the conjecture holds true
  for $g := 6$.  The conjecture was refuted in~\cite{Mehrabian12}.
} %
In Theorem~\ref{thm:no-windmill}, roughly speaking, we prove that a critical graph with girth at least five cannot contain a ``windmill''.  A windmill essentially
consists of a number of shortest paths meeting in an ``axis'', with the paths originating from vertices of degree two and having in their interior only vertices of
degree three.  Theorem~\ref{thm:main} is a corollary of Theorem~\ref{thm:no-windmill}: using an approach inspired by the discharging method from topological graph
theory, we argue that a hypothetical critical graph with girth at least five and average degree less than three always contains a windmill.

For our proof of Theorem~\ref{thm:no-windmill}, we define a class of graphs which we call ``decent'', which have the property that they cannot be contained in a
critical graph.  More importantly, we give a \textit{gluing} operation which preserves ``decency''.  Starting from a small family of basic ``decent'' graphs, by
gluing inductively, this approach allows us to show that certain more complicated configurations cannot be contained in critical graphs, which leads to the proof of
Theorem~\ref{thm:no-windmill}.

This paper is organized as follows.  In the next section, we will discuss some notation as well as basic facts on good edge-labelings.  In
Section~\ref{sec:windmills}, we define windmills, and commence upon the proof of their non-existence.  Section~\ref{sec:gluing} contains the definition of ``decent''
graphs and the gluing mechanism.  Theorem~\ref{thm:no-windmill} is stated and proved in Section~\ref{sec:no-windmill}, and Theorem~\ref{thm:main} is derived in
Section~\ref{sec:proof-discharge}.


\section{Basic facts about good edge-labelings}

We will heavily rely on the above-mentioned characterization of a good edge-labeling using cycles instead of paths.  For this, we use the following definitions.
Let~$H$ be a path or a cycle, and $\phi\colon E(H)\to\RR$ an edge-labeling of~$H$.  Let~$Q$ be a proper sub-path of~$H$ (i.e., a path contained in~$H$ which is not
equal to~$H$) with at least one edge.  For a real number~$\mu$, we say that~$Q$ is a \textit{local minimum with value~$\mu$ in~$H$,} if $\phi(e)=\mu$ for all $e\in
E(Q)$, and for every edge~$e'\in E(H)\setminus E(Q)$ sharing a vertex with~$Q$ we have $\mu < \phi(e')$.

Distinct minima must necessarily be vertex disjoint.  Good edge-labelings can be characterized in terms of local minima of cycles.  We leave the verification of
the following easy lemma to the reader (or see~\cite{MBode-MTh}).

\begin{lemma}
  An edge-labeling~$\phi$ of a graph~$G$ is good, if, and only if, every cycle~$C$ in~$H$ has two local minima.
  \qed
\end{lemma}

Obviously, the property of an edge-labeling being good depends only on the order relation between the labels of the edges.  In particular, scaling (multiplying each
label by a strictly positive constant), and translation (adding a constant to each label) do not change whether a labeling is good or not.

We say that a \textit{$k$-vertex} is a vertex of degree $k$; a $k^-$-vertex is a vertex of degree at most $k$; and a $k^+$-vertex is a vertex of degree at least $k$.

Araujo et al.~\cite{AraujoCohenGiroireHavet11} proved the following property of critical graphs.

\begin{lemma}[\cite{AraujoCohenGiroireHavet11}]\label{lem:acfg}\label{lem:matchingcut}\label{lem:degree-1}
  A critical graph does not contain a matching-cut.
  
  In particular, the minimum degree of a critical graph is at least two, and, unless it is a triangle~$C_3$, it contains no two adjacent 2-vertices.
\end{lemma}

\myparagraph{%
  For the rest of the section,%
} 
let~$G$ be a critical graph other than~$C_3$ and~$K_{2,3}$.  We prove some basic properties of~$G$. 

\begin{lemma}\label{lem:3-cycles}\label{lem:232-loops}
  Let $C$ be a cycle in $G$ whose every vertex has degree at most three.  Then there are two vertices of $C$ with a common neighbour in $G-C$.
\end{lemma}
\begin{proof}
  We proceed by contradiction: let $C'$ be a shortest cycle whose every vertex has degree at most three.  If $G-C'\neq \emptyset$, then it can be easily seen that
  the set of edges with exactly one endpoint in $C'$ forms a matching-cut, contradicting Lemma~\ref{lem:acfg}.  If $G-C' = \emptyset$, then~$G$ is a cycle.  Since~$G
  \ne C_3$, there is a good edge-labeling for this cycle, contradicting the criticality of~$G$.
\end{proof}



\begin{lemma}\label{lem:232-path}
  Let~$P$ be a path of length at least one in~$G$ whose end vertices have degree two and internal vertices have degree at most three.  Then two vertices of~$P$ have
  a common neighbour in~$G-P$.
\end{lemma}
\begin{proof}
  We proceed by contradiction: let~$P'$ be a shortest path between two vertices of degree two with inner vertices of degree three.  If $G-P'\neq \emptyset$, then the
  set of edges with exactly one endpoint in~$P'$ forms a matching cut; contradicting Lemma~\ref{lem:acfg}.  If $G-P' = \emptyset$, then~$P'$ cannot be a shortest
  such path.  (We note that the proof goes through if the length of $P$ is~1.)
\end{proof}




\section{Windmills}\label{sec:windmills}

To motivate the definition of windmills, let us take a look at how they will be used in the proof of Theorem~\ref{thm:main}.
The proof uses a discharging type argument.  We assign ``charges'' to the vertices: vertex $v$ receives charge $6-2d(v)$.  Note that only 2-vertices have positive
charge.  Since the average degree of $G$ is less than $3$, the total charge of the graph is positive.  Now, we ``discharge'' 2-vertices.
Applying Lemma~\ref{lem:232-path}, charges are sent from 2-vertices to $4^+$-vertices via shortest paths consisting of only 3-vertices. 
Later on, we will show that these paths are internally disjoint.  Since no charge is lost during the discharging phase, if after discharging all vertices have
non-positive charge, then we have a contradiction.  However, there may be some vertices with positive charge.  These vertices are the centers of the structures which
we refer to as ``windmills.''
  
\myparagraph{In the remainder of this section,} 
let~$G$ be a critical graph of girth at least five.

\myparagraph{For a tree~$H$ and vertices $x,y$ of $H$, we denote by $xHy$ the unique path between $x$ and~$y$ in~$H$.}
An \textit{internally shortest 3-path} is a path $P=x_0 \dots x_\ell$ with $\ell \ge 1$ and $d(x_j)=3$ for $j \in \{1,\dots,\ell-1\}$, such that, for $e := x_0x_1$,
the path $x_1Px_\ell$ is a shortest path in $G-e$.  In particular, the path $x_1Px_\ell$ is induced in~$G$.  We say that $P$ starts in~$x_0$ and ends in~$x_\ell$.

\begin{remark}\label{rem:2-2-path}
  By Lemma~\ref{lem:232-path}, the graph $G$ has no internally shortest 3-path that starts and ends in 2-vertices.
\end{remark}

So for an edge $x_0x_1$, a \textit{sail} with \textit{tip $x_0x_1$} is defined to be an internally shortest 3-path $P=x_0 x_1 \dots x_\ell$ that starts in a 2-vertex
and ends in a $4^+$-vertex which has minimum length $\ell$ among all such internally shortest 3-paths.

\begin{remark}\label{rem:closest-non3-vtx}
  Among the vertices~$v$ of degree $\deg(v) \ne 3$, the ending vertex $x_\ell$ of a sail is among those which have minimum distance from $x_1$ in $G-e$.  Note that,
  in $G-e$, the vertex~$x_0$ has larger distance from $x_1$ than $x_\ell$: otherwise we would have a contradiction (either from having a $4^+$ vertex closer to~$x_1$
  or by Lemma~\ref{lem:232-loops}).
\end{remark}

\begin{definition}
  Let $k\ge 3$ be an integer, and $y$ a vertex of degree $\max(4,k)$ or $k+1$.
  A \textit{$k$-windmill} with \textit{axis~$y$} in~$G$ is an induced subgraph~$H$ of~$G$, spanned by the union of~$k$ sails beginning in~$k$ distinct tips, and each
  of them ending in~$y$.
  A windmill $H$ is called \textit{complete} in $G$, if it is not a proper subgraph of another windmill.
\end{definition}

Note that it is possible that two sails of the same windmill start in the same 2-vertex (but have different tips).


%
%
%

\begin{lemma}\label{lem:sails}\label{lem:internally-disjoint}
  Let $P=x_0 x_1 \dots x_\ell$ and $P'=x'_0 x'_1 \dots x'_{\ell'}$ be any two sails in $G$. Then 
  \begin{enumerate}[(a)]
  \item if a vertex is adjacent to two vertices of $P$, then one of the two is the starting vertex $x_0$;
  \item if $P$ and $P'$ have distinct tips but identical ending, i.e., $x_\ell=x'_{\ell'}$, 
    then no vertex is adjacent to two (or more) of the vertices of the path $x_1 \dots x_{\ell-1} x_\ell x'_{\ell'-1} \dots x'_1$.
  \item $P$ and $P'$ either share the same tip or they are internally disjoint;
  \end{enumerate}
\end{lemma}

\begin{proof}
  The facts that $G$ has girth at least 5 and that $P$ is a shortest path from its tip $x_0x_1$ to a $4^+$-vertex easily imply (a).
  To prove (b), assume otherwise, that is, there are vertices adjacent to two vertices of the path $Q_1=x_1 \dots x_{\ell-1} x_\ell x'_{\ell'-1} \dots x'_1$.  Of all
  such vertices, let $w_1$ be a vertex whose neighborhood's intersection with $P$, say $x_i$, is closest to $x_1$ on~$P$.  If there are ties, then take $w_1$ to be
  the vertex whose neighborhood's intersection with $P'$, say $x'_j$, is closest to~$x'_1$ on~$P'$.  Since $G$ has girth at least 5, vertex $w_1$ is well defined.
  Now recalling that $P$ and $P'$ are shortest paths from their tips to a $4^+$-vertex, it can be easily seen that $w_1$ must be a 3-vertex (cf.\
  Remark~\ref{rem:2-2-path}).  By Lemma~\ref{lem:232-path} and the choice of $w_1$, two vertices of the path $Q_2=x_1 \dots x_i w_1 x'_j \dots x'_1$ have a common
  neighbor (one of which must be~$w_1$).  Notice that $|Q_2| < |Q_1|$, since the girth or~$G$ is at least~5.  Choose~$w_2$ for~$Q_2$ the same way that we chose~$w_1$
  for~$Q_1$ and we continue the process.  Since the lengths of~$Q_i$'s are decreasing, the process cannot be repeated forever, so at some point we get a
  contradiction to Lemma~\ref{lem:232-path}.

  Item~(c) follows along the same lines as Item~(b).
\end{proof}

Since a windmill contains a unique $4^+$-vertex, it can only be a subgraph of another windmill if their axes coincide, and the larger one has at least one more sail.
The sails of a windmill are pairwise internally disjoint by Lemma~\ref{lem:internally-disjoint} (and the definition of a windmill, by which all tips have to be
distinct).  Moreover, there are no edges between vertices of the same sail, except if one of them is the starting vertex and the other is the axis.  Note that every
edge between the axis and the starting vertex of a sail is itself a sail.

%
%
%
%

\subsection{Flags}

For a fixed complete windmill, we now study vertices that are themselves outside the windmill, but that have two or more neighbors inside the windmill.  We call such
a vertex a \textit{flag.}  We need to classify the flags.  For this, we make the following notational convention.  Let~$H$ be a complete windmill and~$w$ an
\textit{$H$-flag} (i.e., a vertex not in~$H$ that has at least two neighbors in~$H$).  We say that~$w$ has \textit{signature} $(d_0\mid d_1,d_2,\dots)$, if $d(w) =
d_0$, and the neighbors of~$w$ in~$H$ have degrees~$d_i$, $i=1,2,\dots$, listed with multiplicities.  We will conveniently replace sub-lists with an asterix~$*$: For
example, $w$ has signature $(d_0\mid d_1,*)$ if it has degree~$d_0$ and at least one of the neighbors of~$w$ in~$H$ has degree~$d_1$.
Or, similarly, $w$ has signature $(d_0\mid d_1,d_2,d_3,*)$, if~$w$ has degree~$d_0$, at least three neighbors in~$H$, and these three are of degrees~$d_i$, $i=1,2,3$.
We will also replace the degree of~$w$ with a joker: $w$ has signature $(*\mid d_1,d_2,\dots)$, if the neighbors of~$w$ in~$H$ have degrees~$d_i$, $i=1,2,\dots$.

The concept of signature is only needed to reduce the possible occurences of flags to a very small number of cases.  It will not be used beyond this section.


\begin{lemma}\label{lem:simple-signatures}
  Let $H$ be a $k$-windmill.
  The graph $G$ has no flag with either of the following signatures:
  \begin{enumerate}[(a)]
  \item\label{lem:simple-signatures:2-2} $(2\mid 2,*)$,
  \item\label{lem:simple-signatures:3-22} $(3\mid 2,2,*)$
  \item\label{lem:simple-signatures:*33} $(*\mid 3,3,*)$
  \item $(*\mid 3,4^+,*)$
  \item $(3\mid 2,4^+,*)$
  \end{enumerate}
\end{lemma}
\begin{proof}
  Lemma~\ref{lem:acfg} implies that there is no flag with signature $(2\mid 2,*)$.  Lemma~\ref{lem:232-path} gives~(\ref{lem:simple-signatures:3-22}).
  Lemma~\ref{lem:sails}(b) implies that there are no flags with signatures~$(*\mid 3,3,*)$ and~$(*\mid 3,4^+,*)$.  For $(3\mid 2,4^+,*)$, let~$w$ be the flag and
  let~$y$ the axis of the (complete) windmill, and~$x$ the staring vertex of a sail such that $x \sim w \sim y$.  (We use the symbol ``$\sim$'' for the adjacency
  relation in~$G$).  Since $w$ is a 3-vertex, $x w y$ is a sail.  Adding
  the vertex $w$ to $H$ thus gives a larger windmill, contradicting the maximality of~$H$.
\end{proof}


\begin{lemma}\label{lem:windmill:3-23--only-other-sail}
  Let $H$ be a complete windmill with axis $y$ and let $w$ be an $H$-flag of signature $(3\mid 2,3,*)$, so that there are $x,v\in V(H)$ with $d(x) = 2$,
  $d(v) = 3$, and $x\sim w \sim v$.  Then $x$ and $v$ are
  not in the same sail of~$H$.
\end{lemma}
\begin{proof}
  Suppose $x$ and $v$ are on the same sail $P$.  Then $xPv+vw+wx$ is a cycle consisting of 2- and 3-vertices only.  By Lemma~\ref{lem:3-cycles}, there must be a
  vertex~$z$ which is a common neighbor of two vertices on the cycle.  By Lemma~\ref{lem:sails}(b), one of these two is~$w$.  Denote the other by~$u$ and note
  that~$u$ is on $xPv$, but $u\ne x,v$.

  Since $v$ has degree 3, $z$ must also have degree 3, because~$P$ is a sail starting in~$x$, and $z$ having degree different from~3 would contradict the minimality
  of the distance from the tip to the end-vertex (by Remark~\ref{rem:closest-non3-vtx}).  Now consider any sail with tip $xw$.  Since~$w$ is of degree 3, it must
  contain either $v$ or~$z$.  But~$v$ cannot be contained in such a sail, because otherwise it would have a non-empty interior intersection with~$P$, contradicting
  Lemma~\ref{lem:sails}(c).  It follows that there is a sail~$Q$ with tip $xw$ containing~$z$ such that
  \begin{equation}\label{star:o238nrsli}\tag{$*$}
    \sabs{wz+zQ} < \sabs{wv+vPy}.
  \end{equation}
  However, the length of~$P$ is at most the length of $xPu+uz+zQ$, and the inequality must be strict, because otherwise the sail $xPu+uz+zQ$ would have a non-empty
  interior intersection with the sail~$Q$.  Thus, it follows that
  \begin{equation}\label{star:o238nrsli2}\tag{$**$}
    \sabs{uPy} < \sabs{uz+zQ}.
  \end{equation}
  Now \eqref{star:o238nrsli} and~\eqref{star:o238nrsli2} together imply that $\sabs{uPy} < 1 + \sabs{zQ} < 1 + \sabs{vPy}$, from which we conclude $\sabs{uPy} <
  \sabs{vPy}$, a contradiction to the fact that~$u$ is between $v$ and~$x$.
\end{proof}

The remaining cases are more complex.  We start with the following fact.

\begin{lemma}\label{lem:windmill:*-23-other--adj-to-axis}
  Let $w$ be an $H$-flag with signature $(*\mid 2,3,*)$,
  so that there are $x,v\in V(H)$ with $d(x) = 2$, $d(v) = 3$, and $x\sim w \sim v$.  If $x$ and $v$ are not in the same sail of~$H$, then $v$ is adjacent to the
  axis~$y$ of~$H$.
\end{lemma}
\begin{proof}
  By Lemma~\ref{lem:simple-signatures}(\ref{lem:simple-signatures:2-2}), $w$ has degree at least $3$.  Suppose that $v$ is on the sail~$P'$ of~$H$ with starting
  vertex $x'$, and $v\not\sim y$.  (Note that $x'\ne x$.)

  On one hand, if the degree of $w$ is $3$, then we have an internally shortest path $x'P'v+vw+wx$ ending in a 2-vertex, contradicting Remark~\ref{rem:2-2-path}.  On
  the other hand, if the degree of $w$ is $4^+$, then $x'P'v+vw$ is a path shorter than~$P'$, but ends in a vertex not of degree $3$, contradicting
  Remark~\ref{rem:closest-non3-vtx}.
\end{proof}

\begin{lemma}
  No flag can have signature $(3\mid *)$.
\end{lemma}
\begin{proof}
  The cases $(3\mid 2,2)$, $(3\mid 4^+,*)$ and $(3\mid 3,3)$ are dealt with in Lemma~\ref{lem:simple-signatures}.

  Let us consider $(3\mid 2,3)$.  By Lemmas~\ref{lem:windmill:3-23--only-other-sail} and~\ref{lem:windmill:*-23-other--adj-to-axis}, denoting the flag by~$w$, the
  2-vertex by~$x$, the 3-vertex by~$v$, and by~$x'$ the starting vertex of the sail~$P'$ containing~$v$, the start and end vertices of the path $Q := x'P'v+vw+wx$
  have degree~$2$.  By Lemma~\ref{lem:232-path}, there is a vertex~$z$ which is a common neighbor of two vertices on~$Q$.  By
  Lemma~\ref{lem:simple-signatures}(\ref{lem:simple-signatures:*33}), one of the two must by~$w$.  Denote the other one by~$u$.

  If~$z$ has degree~$4$ or more, then the length of $x'Qu +uz$ is shorter than that of~$P'$, contradicting Remark~\ref{rem:closest-non3-vtx}.
  
  If~$z$ has degree~$3$, then the length of $x'Qu+uz+zw+wx$ is at most the length of~$P'$ (because $G$ has no $C_3$ or $C_4$), contradicting
  Remark~\ref{rem:closest-non3-vtx}.
\end{proof}

\begin{lemma}\label{lem:windmill:4-23-same--adj-to-axis}
  Let $w$ be an $H$-flag with signature $(4^+\mid 2,3,*)$,
  so that there are $x,v\in V(H)$ with $d(w)=4^+$, $d(x) = 2$, $d(v) = 3$, and $x\sim w \sim v$.  If $x$ and $v$ are on the same sail of~$H$, then $v$ is
  adjacent to the axis $y$ of~$H$.
\end{lemma}
\begin{proof}
  Denote the sail containing both $x$ and~$v$ by~$P$.
  If $v$ is not adjacent to the axis~$y$, then $xPv+vw$ is a path shorter than $P$ that ends in a vertex not of degree $3$, contradicting
  Remark~\ref{rem:closest-non3-vtx}.
\end{proof}


We conclude the subsection with the following important consequence of our investigation of flags.  Per se, windmills are subgraphs of~$G$ induced by sails, but the
next lemma shows that in a complete windmill, every edge already belongs to some sail.
		
\begin{lemma}\label{lem:windmills:no-x-edges}
  All edges of a complete windmill are on sails.
\end{lemma}
\begin{proof}
  To argue by contradiction, consider an edge $e=uv$ whose vertices are on a windmill, but which is not on a sail.  By Lemma~\ref{lem:acfg}, vertices $u$ and $v$
  cannot be both 2-vertices.  By Lemma~\ref{lem:sails}, 
  vertices $u$ and $v$ cannot be both 3-vertices on the same sail.  The same lemma, shows a 3-vertex and the axis cannot be the endpoints of $e$.  Now let $u$ be a
  2-vertex and $v$ be a 3-vertex.  Vertices $u$ and $v$ can either be in the same sail or not.  The former contradicts Lemma~\ref{lem:3-cycles} and
  Lemma~\ref{lem:sails}.  For the latter case, let $P$ be a sail which contains a 3-vertex~$v$ which is adjacent to a 2-vertex~$u$ on another sail.  Let $x$ be the
  2-vertex of the sail $P$.  Then $xPv+vu$ is an internally shortest path, contradicting Lemma~\ref{lem:sails}.

  The only remaining case is when $u_1$ and $u_2$ are two 3-vertices on distinct sails $P_1$, $P_2$, of the windmill.  Say $u_1 \in P_1$ and $u_2 \in P_2$, and that
  the starting vertex of~$P_i$ is~$x_i$ (note that $x_1=x_2$ is possible).  Denote by $Q := x_1P_1u_1+u_1u_2+u_2Px_2$ the path (or cycle) starting with the tip of
  one of the two sails, taking the edge~$e$, and ending in the tip of the other -- the two starting vertices~$x_i$ of the sails may coincide, in which case~$Q$ is a
  cycle.
  
  By Lemma~\ref{lem:232-path} (or Lemma~\ref{lem:3-cycles} if~$Q$ is a cycle), there must be a vertex $w$ which is adjacent to two vertices $y_1,y_2$ on~$Q$.  The
  vertex~$w$ cannot be the axis as it contradicts Remark~\ref{rem:closest-non3-vtx} or entails that there is a triangle.  Hence, $w$ is a flag.
  
  By Lemma~\ref{lem:simple-signatures}(\ref{lem:simple-signatures:*33}), one of the~$y_i$ must be a 2-vertex, the other may be a 2- or a 3-vertex.  If, say $y_2$ is
  a 3-vertex, then either by Lemma~\ref{lem:windmill:*-23-other--adj-to-axis} or Lemma~\ref{lem:windmill:4-23-same--adj-to-axis},
  $y_2$ must be adjacent to the axis --- a contradiction, since the only vertex on~$Q$ which might be adjacent to the axis has degree at most three, and two
  neighbors on~$Q$, the third is the axis (and~$w$ is not the axis).
  
  Hence, we conclude that $y_1$ and $y_2$ are both 2-vertices (in particular, $x_1\ne x_2$).  Since, by what we have just said, $w$ is not adjacent to a third vertex
  of~$Q$, by Lemma~\ref{lem:3-cycles}, there is another vertex $w'$ which is adjacent to two 3-vertices on~$Q$.  But such a vertex would be a flag with
  signature~$(*\mid3,3,*)$, which is impossible by Lemma~\ref{lem:simple-signatures}(\ref{lem:simple-signatures:*33}).
\end{proof}

\subsection{The flag graph} 

We have narrowed down the possible configurations involving flags of windmills.  To summarize the results above, a flag can be adjacent to
\begin{itemize}
\item several 2-vertices on the tips of sails,
\item and at most one of the following:
  \begin{itemize}
  \item the axis, or
  \item one 3-vertex which is adjacent to the axis on a sail.
  \end{itemize}
\end{itemize}
Moreover,
\begin{itemize}
\item only one flag can be adjacent to the axis,
\item every 3-vertex as above can be adjacent to at most one flag (obviously),
\item there are no edges except those in the sails of the windmill or incident to the flags.
\end{itemize}

This structure can be nicely dealt with in an inductive manner (rather than delving into a humongous list of case distinctions).  In the remainder of this section,
we show how the structure of flags on windmills can be modelled by a directed graph which we call \textit{flag graph}, which has a tree-like structure.
The possibility of a flag which is adjacent to the axis is a complication.  Such a flag, if existent, is omitted from the construction of the flag graph.


Let~$W$ be a complete windmill contained in~$G$.  A flag which is adjacent to the axis of~$W$ is called \textit{irregular} (recall that there can be at most one);
the other flags are called \textit{regular.}
It is important to realize that a sail whose tip is adjacent to an irregular flag has length at least~3, because the girth of~$G$ is at least~5.

The flag graph $F = F(W,G)$ of a windmill~$W$ is a directed bipartite graph.  One side of the bipartition of the vertex set of~$F$ comprises the
\textit{flag-vertices,} which are in one-to-one correspondence with the regular flags of~$F$.  The other side of the bipartition consists of the
\textit{sail-vertices,} which are in one-to-one correspondence with the 2-vertices at the tips of the sails of~$W$.  We say that a sail-vertex which corresponds to a
2-vertex contained in two sails is \textit{degenerate;} a sail-vertex corresponding to a 2-vertex contained in only one sail is called \textit{non-degenerate.}
There are two types of arcs:
\begin{itemize}
\item[\it 2-arc] Whenever a regular flag~$w$ is adjacent to a 2-vertex on the tip of a sail~$P$ of~$W$, we have an arc from the sail-vertex corresponding
  to~$P$ to the flag-vertex corresponding to~$w$.  Note that, in this case, the sail-vertex is non-degenerate.
\item[\it 3-arc] Whenever a regular flag~$w$ is adjacent to a 3-vertex~$x$ in~$W$, then there is an arc from the flag-vertex corresponding to~$w$ to the sail-vertex
  which represents the sail containing~$x$.
\end{itemize}
The flag graph may contain anti-parallel arcs: A regular flag might be adjacent to a 2- and a 3-vertex of the same sail.

We note the following observations which follow directly from the construction and the earlier results of this section (see the summary above).

\begin{lemma}\label{lem:windmills:flag-graph-properties}
  Let~$G$ contain a windmill~$W$, and let~$F=F(W,G)$ be its flag graph.
  \begin{enumerate}[(a)]
  \item Degenerate sail-vertices have out-degree~0; non-degenerate sail-vertices have out-degree at most~1.
  \item Flag-vertices have out-degree at most~1.
  \item Degenerate sail-vertices have in-degree at most~2; non-degenerate sail-vertices have in-degree at most~1.
  \item Flag-vertices have in-degree at least~1.
  \end{enumerate}
  Moreover, the undirected connected components of~$F$ are in one-to-one correspondence with the blocks in a block-decomposition of~$H$. 
In particular, only non-degenerate sail-vertices can be contained in directed cycles.
\qed
\end{lemma}

We now show how the flag graph can be constructed inductively from basic elements and construction rules.  
\begin{enumerate}[$\bullet$]
\item\makebox[17ex][l]{\bf Basic element~S}       A single non-degenerate sail-vertex.
\item\makebox[17ex][l]{\bf Basic element~S$_-$}   A single degenerate sail-vertex.
\item\makebox[17ex][l]{\bf Basic element~S$^+$}   A flag-vertex and a sail-vertex with an arc from the sail-vertex to the flag-vertex.
\item\makebox[17ex][l]{\bf Basic element~C2}      A cycle of length~2.
\item\makebox[17ex][l]{\bf Basic element~C4}      A cycle of even length at least~4.
\end{enumerate}
If~$F$ is a flag graph, it can be extended with the following construction rules:
\begin{enumerate}[$\bullet$]
\item\makebox[20ex][l]{\bf Construction rule~U} Start a new connected component (in the undirected sense) by adding one of the basic elements without connection
  to~$F$.
\item\makebox[20ex][l]{\bf Construction rule~A} Add a sail-vertex with an outgoing 2-arc linking it to an arbitrary flag-vertex of~$F$.
\item\makebox[20ex][l]{\bf Construction rule~B} Add a flag-vertex and a sail-vertex, together with a 2- and a 3-arc: the 3-arc goes from the new flag-vertex to an
  arbitrary sail-vertex in~$F$, and the 2-arc goes from the new sail-vertex to the new flag-vertex.
\end{enumerate}


\begin{lemma}\label{lem:windmills:flag-graph-inductive}
  Let~$G$ contain a windmill~$W$, and let~$F=F(W,G)$ be its flag graph.  Then~$F$ can be constructed using the above basic elements and construction rules.
\end{lemma}

\begin{proof}
  We show that each connected component of the flag graph can be constructed in an inductive manner as follows.  Let $C$ be a connected component of the flag graph.
  Assume that $C$ cannot be obtained in one step by Construction rule U.  If $C$ has a sail vertex $s$ with out-degree 1 and in-degree 0 whose flag neighbour has
  either out-degree or in-degree at least 2, then construct $C\backslash s$ first and then obtain $C$ by applying Construction rule A (notice that by the definition,
  the flag neighbour of $s$ must be in $C\backslash s$ as $C$ cannot be constructed by Construction rule U).  Otherwise, $C$ has a sail vertex $s$ with out-degree 1
  and in-degree 0 whose flag neighbour $f$ has both out-degree and in-degree 1.  In this case, construct $C-\{s,f\}$ first and then obtain $C$ by applying
  Construction rule B.
\end{proof}

In Section~\ref{sec:no-windmill}, this construction will be used to inductively ``glue together'' the graph induced by a windmill and its flags.  Before we can do
that, in the following section, we explain the gluing operation.

\section{Typed graphs and gluing}\label{sec:gluing}

Let $P$ be a path and~$Q$ a local minimum with value~$\mu$ in~$P$.  We say~$Q$ is an \textit{imin} if $\mu < 0$ or~$Q$ contains no endvertices of~$P$.

\begin{definition}\label{def:typed-graph-decent}
  A \textit{typed graph with types~$\tau$} is a graph together with a mapping $\tau\colon V(G) \to \{0,1,2\}$.  In other words, every vertex has one of three
  possible types: it can be either a type-0, a type-1, or a type-2 vertex.  The figures in this section show graphs with the types of the vertices in square
  brackets.

  A \textit{decent labeling} of a typed graph is a good edge-labeling with the following properties:
  \begin{enumerate}[(a)]
  \item If $P$ is a path between two type-2 vertices, then the length of~$P$ is at least three, and at least one of the following two conditions hold:
    \begin{enumerate}[({a}.1)]
    \item there is an imin on~$P$ such that between each endpoint of~$P$ and this imin, there is an edge with strictly positive label;
    \item there are (at least) two imins.
    \end{enumerate}
  \item If $P$ is a path between between a type-1 vertex~$v$ and a type-2 vertex~$w$, then the length of~$P$ is at least two, and
    \begin{enumerate}[({b}.1)]
    \item there is an imin on~$P$ which does not contain~$v$.
    \end{enumerate}
  \end{enumerate}
  A typed graph is \textit{decent} if it has a decent edge-labeling.
\end{definition}

Note that if a path~$P$ satisfies either of the conditions (a.1) or~(a.2), then~$P$ also satisfies~(b.1).  Moreover, if a typed graph $G$ has no type-2 vertex, then
any good edge-labeling of $G$ is also decent.

For $t\in\{1,2\}$, a path in a typed graph is called \textit{$t$-simple,}\label{text:tsimple} if the type of every interior vertex is strictly less than~$t$.  We
leave it to the reader to convince himself that, in order to verify that a good edge-labeling is decent, it suffices to check 2-simple paths in~(a), and 1-simple
paths in~(b), respectively.


Before we continue discussing typed graphs in general, we discuss several examples which we will need in the remainder of the paper: We describe typed graphs and
define concrete decent edge-labelings on them (where they exist).  The graphs are indeed \emph{rooted graphs} with the root denoted by~$y$ --- this is owed to the
fact that we will later apply them to windmills, with the root corresponding to the axis of the windmill.  Hence, we will discuss multiple versions of some graphs,
with the difference lying only in the location of the root vertex~$y$.


\begin{example}\label{exp:gluing:path-w-flag}\textit{Decent labeling of typed paths ending in a type-2 vertex.}  
  Let~$P$ be a path of length at least two with root vertex~$y$ as shown in Fig.~\ref{fig:gluing:path-w-flag}, and ending in a type-2 vertex~$w$.  Denote the
  vertex adjacent to~$w$ by~$x$.  The root~$y$ is type-0 or type-1, $x$ is type-0, and all remaining vertices of~$P$ are type-1 vertices.  Then~$P$ is a decent typed
  graph.  Fig.~\ref{fig:gluing:path-w-flag} shows a decent edge-labeling~$\phi$.  If~$P$ has length two, then $\phi(wx)= -1$ and $\phi(xy) = \nfrac34$.  If~$P$
  has length at least three, then let $\phi(wx)= -1$, $\phi(v_{n-1}v_n) = {17}/{24}$ and the rest of the edges have label $+1$.
  \begin{figure}[htp]
    \centering%
    \hspace*{-1em}\scalebox{.5}{\input{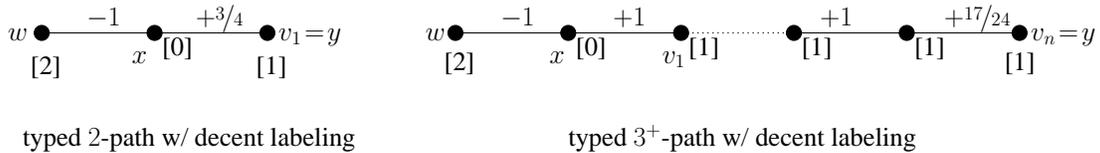}}%
    \caption{Typed paths with decent edge-labelings}\label{fig:gluing:path-w-flag}%
  \end{figure}
\end{example}

\begin{example}\label{exp:gluing:path-wo-flag}\textit{Decent labeling of typed paths ending in a type-1 vertex.}
  Let~$P$ be a path of length at least three with root~$y$ as shown in Fig.~\ref{fig:gluing:path-wo-flag}.  The type of the root vertex is 0 or 1, and all other
  vertices have type 1.  Then~$P$ is a decent typed graph.  Fig.~\ref{fig:gluing:path-w-flag} shows a decent edge-labeling.  The additional edges (dots) have
  label~$+1$.
  \begin{figure}[htp]
    \centering%
    \scalebox{.5}{\input{typed-path-noflag.pstex_t}}%
    \caption{Typed paths with decent edge-labelings}\label{fig:gluing:path-wo-flag}%
  \end{figure}
\end{example}


\begin{example}\label{exp:gluing:degsail}\textit{Decent labelings of cycles without type-2 vertices.}
  Let~$C$ be a cycle on type-0 and type-1 vertices as shown in Fig.~\ref{fig:gluing:degsail}.  Let $u$ be a type-0 vertex and all remaining vertices of~$C$ be
  type-0 or type-1 vertices.  Then~$C$ is a decent typed graph.  Fig.~\ref{fig:gluing:degsail} shows decent edge-labelings.  If~$C$ has length five, then based on
  the position of vertex $y$, we present two different decent labeling for later applications, either of which is a decent edge-labeling of the typed 5-cycle $C$
  independent of the position of $y$.
  \begin{figure}[htp]%
    \centering%
    \scalebox{.5}{\input{degenerate-sail.pstex_t}}\\
    \caption{Typed cycles with decent edge-labelings}\label{fig:gluing:degsail}%
  \end{figure}
\end{example}


\begin{figure}[htp]%
  \centering%
  \scalebox{.5}{\input{fish.pstex_t}}\\
  \scalebox{.5}{\input{wheel-segment.pstex_t}}\\
  \scalebox{.5}{\input{missing-case.pstex_t}}\\
  \caption{Cycles having a type-2 vertex with decent edge-labelings.}\label{fig:gluing:fishAndCycle}%
\end{figure}
\begin{example}\label{exp:gluing:fish}\textit{Decent labelings of cycles with one type-2 vertex, part~I.}
  Consider a typed graph consisting of a cycle and one extra edge with one end on the cycle as shown in the top part of Fig.~\ref{fig:gluing:fishAndCycle}.  The
  root~$y$ is off the cycle and has type 0 or 1.  There is a type~2 vertex on the cycle, and it is adjacent to two type~0 vertices.  The rest of the vertices have
  type~1.  The figure shows a decent labeling.
\end{example}

\begin{example}\label{exp:gluing:alphabetaCycle}\textit{Decent labelings of cycles with one with type-2 vertex, part~II.}
  Consider a typed graph consisting of a cycle and one extra edge with one end on the cycle as shown in the middle part of Fig.~\ref{fig:gluing:fishAndCycle}.  The
  root~$y$ is on the cycle, and has type 0 or 1.  There is a type-2 vertex on the cycle, and it is adjacent to two type-0 vertices.  The rest of the vertices have
  type~1.  For given $\alpha,\beta$ satisfying $\nfrac23 < \alpha < \nfrac34 < \beta$, a decent labeling can be constructed, as shown in the figure.
\end{example}


\textbf{NEW TODO: CHECK!}
\begin{example}\label{exp:gluing:alphaCycle}\textit{Decent labelings of cycles with one with type-2 vertex, part~III.}
  Consider a typed graph consisting of a cycle as shown in the bottom part of Fig.~\ref{fig:gluing:fishAndCycle}.  The root~$y$ is on the cycle, and has type 0 or~1.
  There is a type-2 vertex on the cycle, and it is adjacent to two type-0 vertices.  The rest of the vertices have type~1.  For given $\alpha, \gamma$ satisfying
  $\nfrac23 < \alpha < \nfrac34$ and $\gamma<0$, a decent labeling can be constructed, as shown in the figure.
\end{example}



\paragraph{%
  A \textit{$k$-wheel}%
} is a typed graph~$H$ which is the union of a cycle and a center vertex connected to $k \ge 2$ of the vertices on the cycle, called \textit{anchors.}  The distance
on the cycle of any two anchors must be at least three.  Fix an orientation of the cycle.  A successor of an anchor is called a \textit{bogey;} the successor of a
bogey is called a \textit{spectator;} all other vertices on the cycle are called \textit{boobies.}  A path contained in the cycle connecting successive anchors is
called a \textit{segment.}  The vertices are to have the following types:
\begin{itemize}
\item anchors and spectators have type~0;
\item the bogies have type~2;
\item the boobies have type~1; and
\item the center vertex either type~0 or type~1.
\end{itemize}

We divide the class of wheels into 3 subclasses: Benign wheels, almost evil wheels, and evil wheels.  The first two kinds are decent, while the third is not.


\begin{example}\label{exp:gluing:wheels}\textit{Decent labelings of benign wheels.}  %
  Consider a wheel in which the center is a type-1 vertex but contains at least one pair of consecutive anchors whose distance is at least four (this is the
  ``benign'' segment of the wheel).  These typed graphs are called \textit{benign wheels}.  Fig.~\ref{fig:gluing:wheels} shows decent edge-labelings of benign
  wheels.
  \begin{figure}[htp]%
    \centering%
    \scalebox{.5}{\input{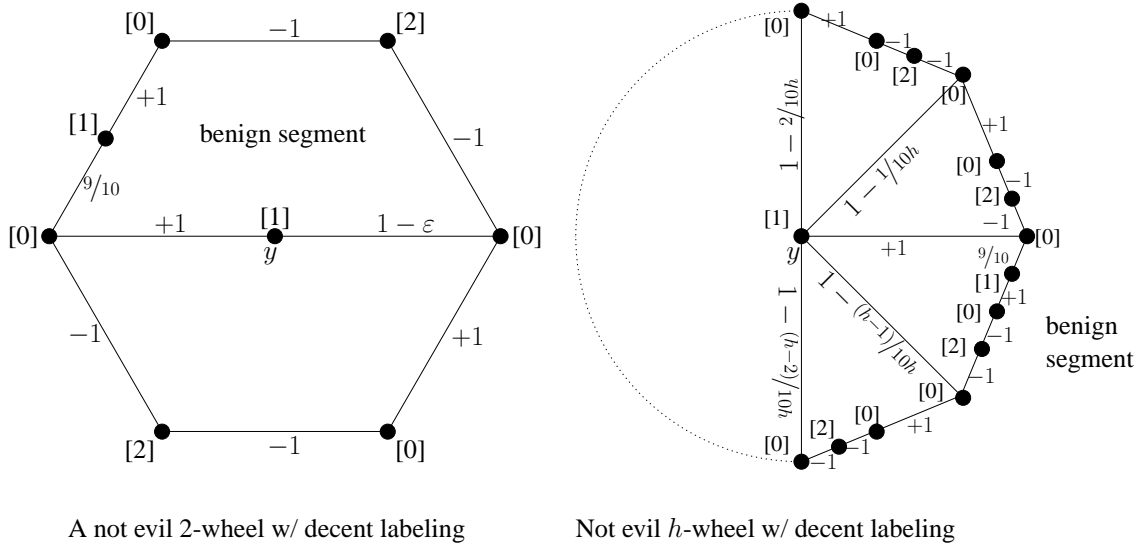}}\\
    \caption{Benign wheels are decent.}\label{fig:gluing:wheels}%
  \end{figure}  
\end{example}


\begin{example}\label{exp:gluing:almost-evil-wheels}\textit{Decent labeling of almost evil wheels.}  %
  Consider a wheel in which the distance of every pair of consecutive anchors is exactly three and the center is a type-0 vertex.  These typed graphs that are only
  different from evil wheels in the type of the center, are called \textit{almost evil wheels}.  Fig.~\ref{fig:gluing:almost-evil-wheels} shows decent
  edge-labelings of almost evil wheels.
  \begin{figure}[htp]
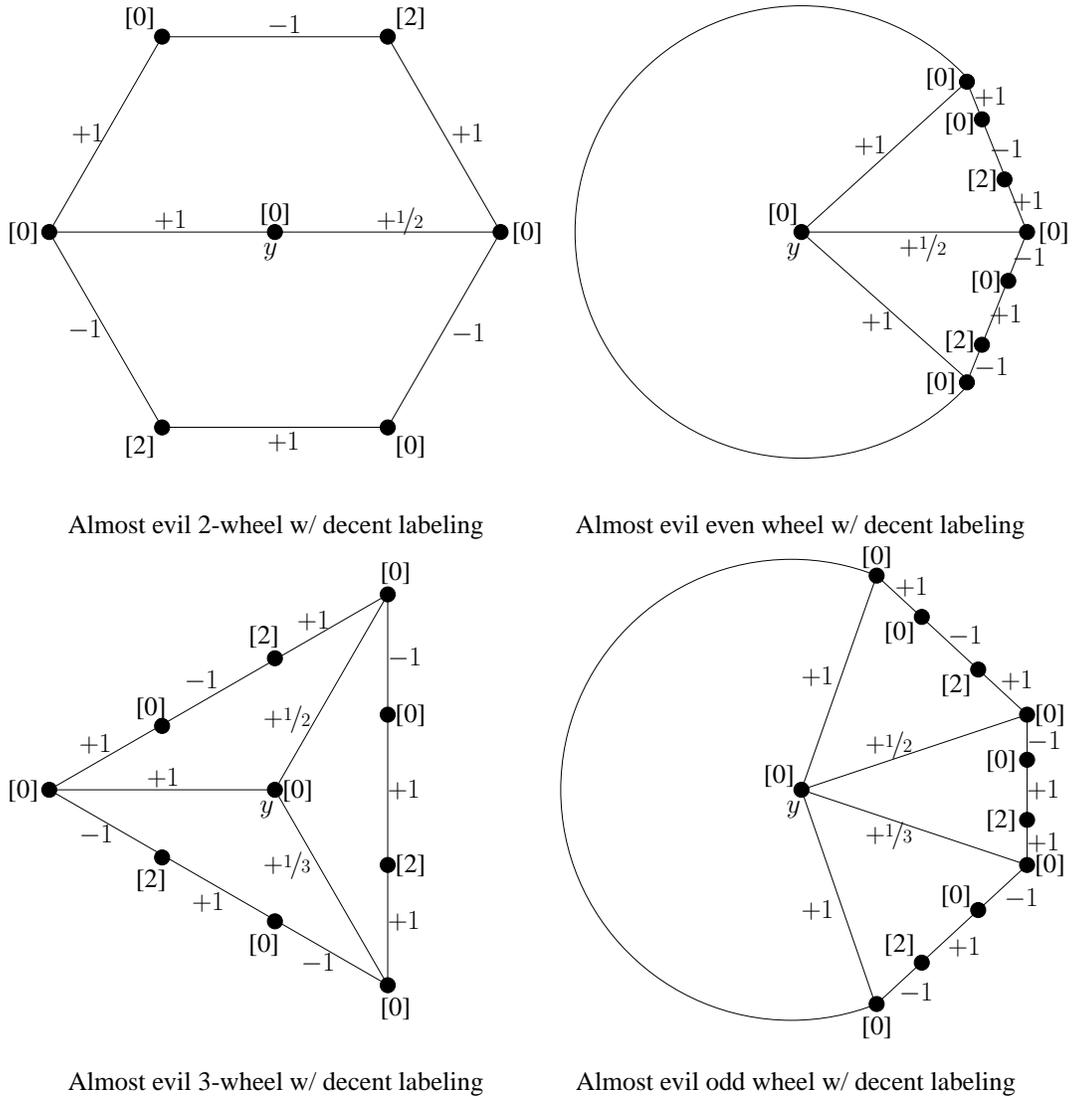
%
    \centering%
    \scalebox{.5}{\input{almost-evil-even-wheel.pstex_t}}\\
    \scalebox{.5}{\input{almost-evil-odd-wheel.pstex_t}}%
    \caption{Almost evil wheels with decent labelings}\label{fig:gluing:almost-evil-wheels}%
  \end{figure}
\end{example}


\begin{example}\label{exp:gluing:evil-wheels}\textit{Evil wheels are not decent.}
  If the distance of every pair of consecutive anchors in a wheel is exactly three and the center is a type-1 vertex, as shown in Fig.~\ref{fig:gluing:evil-wheels},
  then wheel is called an \textit{evil wheel.}  It can be easily seen that evil wheels have a good edge-labeling.  However, they are not decent.
  \begin{figure}[htp]%
    \centering%
    \scalebox{.333333333333}{\input{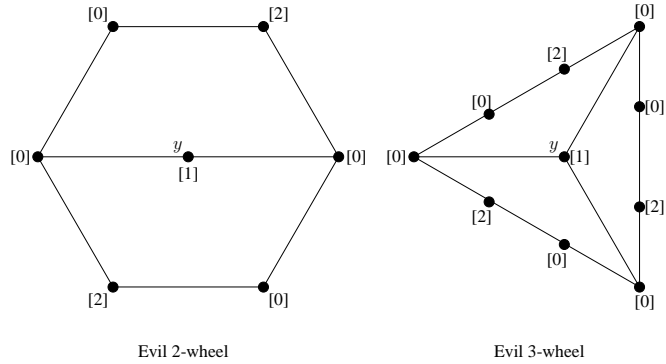}}\\
    \caption{Evil wheels are \textit{not} decent.}\label{fig:gluing:evil-wheels}%
  \end{figure}  
\end{example}

\begin{remark}\label{remark:evil-decent}
  It can be seen that if a wheel is not evil then examples \ref{exp:gluing:almost-evil-wheels} and~\ref{exp:gluing:wheels} can yield a decent edge-labeling.  In
  other words, if a wheel is not evil, then it is decent.
\end{remark}


\subsection{Swell subgraphs}

Lemma~\ref{lem:fundamental} below is the fundamental motivation behind defining typed graphs and swell graphs.

\begin{definition}\label{def:swell-graph-decent}
  Let $H$ be a proper subgraph of a graph~$G$.  We say that \textit{$H$ is a swell subgraph of $G$}, if~$H$ is typed with at least one type~0 or type~1 vertex and
  the following properties:
  \begin{enumerate}[(a)]
  \item no type~0 vertex in $H$ has a neighbor in $G-H$;
  \item every type~1 vertex in $H$ has at most one neighbor in $G-H$;
  \item no vertex in~$G-H$ has two or more type~1 vertices of~$H$ as neighbors. 
  \end{enumerate}
\end{definition}

The shaded area in Fig.~\ref{fig:gluing:swell-subgraph} is an example of a swell subgraph.

\begin{figure}[htp]
  \centering
  \scalebox{.5}{\input{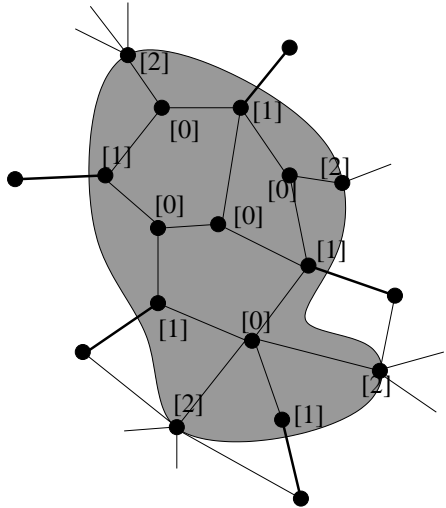}}
  \caption{A swell subgraph}\label{fig:gluing:swell-subgraph}
\end{figure}

\begin{lemma}\label{lem:fundamental}
  Let $H$ be a decent typed graph.  A critical graph cannot contain $H$ as a swell subgraph.
\end{lemma}

In the following lemma, we use the shorthand $-\infty$ to denote a negative number whose absolute value is larger than all other, ``finite'', absolute values.

\begin{proof}[Proof of Lemma~{lem:fundamental}]
  Assume otherwise and let $H$ be a decent typed graph, which is a proper subgraph of a critical graph $G$.  We prove that $G$ has a good edge-labeling.  Define the
  graph~$G'$ by deleting from~$G$ all the type-0 and type-1 vertices of~$H$.  Note that $E(H)\cap E(G') = \emptyset$, by Definition~\ref{def:typed-graph-decent}.
  Since $H$ has at least one type~0 or type~1 vertex, $G'$ has a good edge-labeling.

  Note that the edges in $M := E(G)\setminus (E(H)\cup E(G'))$ are incident to type-1 vertices of~$H$.  Now take a decent labeling of~$H$ and scale it so that all
  nonzero labels have absolute value at least~$2$.  Also, take a good labeling of $G'$ and scale it so that all labels have absolute value at most~$1$.  We combine
  these two to form a labeling of the edges of~$G$, where the edges in $M$ receive the label $-\infty$.
  We prove that this forms a good edge-labeling of $G$.  Consider a cycle~$C$ in $G$.  If $E(C) \subset E(G')$ or $E(C) \subset E(H)$, then $C$ has two local minima.
  
  Otherwise, consider the graph $C \cap H$.  Its connected components are path, at least one of which must have non-zero length, so let~$P$ be such a component.
  Denote the end-vertices of~$P$ by $x,y$.

  Notice that with the above mentioned relabeling, an imin on $P$ is in fact a local minimum on $C$.

  First, assume that both $x$ and $y$ are of type~2.  If $P$ has at least two imins, then those two are in fact two local minima in $C$.  Otherwise $P$ has an imin
  such that between each endpoint of $P$ and this imin, there is an edge with strictly positive label.  Moreover, by the scaling of labels in $H$, these two labels
  have value at least 2.  Considering the scale of the labels in $G'$, there is a local minimum of $C$ that belongs to $G'$.  This local minimum in addition to the
  imin on $P$ are two local minima of $C$.

  If $x$ has type~2 and $y$ has type~1, then the edge~$e$ of~$C\setminus P$ adjacent to~$y$ has label $-\infty$.  Moreover, by the definition of a decent labeling,
  there is an imin on~$P$ which is not incident to~$e$.  By the same argument as above, this imin in~$P$ is a local minimum in~$C$.  Hence, we have two local minima
  on~$C$.

  Finally, if both~$x$ and~$y$ have type~1, let~$e$ and~$f$ be the edges of $C\setminus P$ incident to~$x$ and~$y$, respectively.  By the definition of a swell
  subgraph, $e$ and $f$ cannot be adjacent.  So $e$ and $f$ are local minima of~$C$ as their labels are~$-\infty$,
\end{proof}


\subsection{Gluing}

In order to use decent typed graphs in inductive arguments, we have the following construction which allows to ``glue'' two decent typed graphs and obtain a new one.

\begin{definition}
  Let~$G_1$ and $G_2$ be typed graph with types~$\tau_i$, $i=1,2$, let~$H$ be an induced subgraph of both $G_1$ and~$G_2$, and $V(H) = V(G_1)\cap V(G_2)$.  We say
  that the typed graph~$G$ with types~$\tau$ is the result of \textit{gluing $G_1$ and~$G_2$ along~$H$,} if $V(G) = V(G_1)\cup V(G_2)$, $E(G) = E(G_1)\cup E(G_2)$,
  and
  \begin{equation*}
    \tau(v) =
    \begin{cases}
      \tau_1(v), &\text{ if $v\in V(G_1)\setminus V(G_2)$,}\\
      \tau_2(v), &\text{ if $v\in V(G_2)\setminus V(G_1)$,}\\
      \min(\tau_1(v),\tau_2(v)), &\text{ if $v\in V(G_1)\cap V(G_2)$.}
    \end{cases}
  \end{equation*}
\end{definition}

We wish to have conditions which ensure that if $G_1$ and~$G_2$ are decent, then~$G$ is, too.  As a motivating example, the reader might want to verify the following
fact (which we do not need in this paper):
\begin{lemma}\label{lem:toy-gluing}
  If for all $v\in V(H)$ we have $\tau_1(v) = \tau_2(v) = 2$, and if $G_1$ and $G_2$ are decent, then~$G$ is decent.
\end{lemma}

Our aim is to decompose windmills into elementary parts---indeed, all parts we need have been discussed in Examples
\ref{exp:gluing:path-w-flag}--\ref{exp:gluing:alphabetaCycle} and \ref{exp:gluing:almost-evil-wheels}--\ref{exp:gluing:wheels}.  For this, we need a considerably
more powerful gluing mechanism than that of Lemma~\ref{lem:toy-gluing}.  We define the class of ``gluable'' typed graphs, which can be glued to each other by 1- and
2-sum operations.

We need to first classify certain special type-2 vertices.
\begin{definition}\label{def:lock}
  For a given quadruple $(G,\tau,\phi,y)$ consisting of a typed graph $G$ with types~$\tau$, a decent edge labeling~$\phi$ of~$G$, and a root vertex~$y$ of~$G$ we
  say that a type-2 vertex $w$ is \textit{locked} if the distance~$d_G(w,y)$ between~$w$ and~$y$ is two, the (unique) path~$P$ between $w$ and~$y$ of length two has
  an imin, and the edge incident to~$y$ on~$P$ has label in $\lt]\nfrac23,\nfrac34\rt[$.

  We call $P$ the \textit{locking path} of $w$.  If $w$ is not locked, we call it \textit{connectable}.
\end{definition}

Now we are ready to give the complete definition of the gluing operation.

\begin{definition}\label{def:gluable}
  We say that a quadruple $(G,\tau,\phi,y)$ consisting of a typed graph~$G$ with types~$\tau$, a decent edge-labeling~$\phi$ of~$G$, and a root vertex~$y$ of~$G$ is
  \textit{gluable}, if the following conditions hold.
  \begin{enumerate}[(a)]
  \item Every path $y,v_1,v_2$ of length~$2$ originating from~$y$ and containing a type-1 vertex~$v_1$ and a vertex~$v_2$ of type 0 or~1 is \textit{admissible:}
    With $\alpha := \phi(y v_1)$ and $\beta := \phi(v_1v_2)$, we have
    \begin{equation*}
      \nfrac23 < \alpha \le \nfrac34 < \beta.
    \end{equation*}
  \item Every 1-simple path\footnote{%
      Recall the definition of $t$-simple from page~\pageref{text:tsimple}.} %
    of length at least one between a type-1 vertex and~$y$ contains an edge with value at least~$\nfrac23$.
  \item Not type-2 vertex is adjacent to the root~$y$.
  \item Let~$w$ be a type-2 vertex in~$G$.  If the distance~$d_G(w,y)$ between~$w$ and~$y$ is two, then every 2-simple path~$P$ between~$w$ and~$y$
  	except for the locking path of $w$, if it exists, satisfies one of the following:
    \begin{quote}
    \begin{itemize}
    \item[(d2.i)] $P$ has an imin, and the edge incident to~$y$ on~$P$ has label at least~$\nfrac34$; or
    \item[(d2.ii)] The edge incident to~$y$ on~$P$ is a local minimum with value in $\lt]0,\nfrac12\rt]$.
    \end{itemize}
    \end{quote}
    If the distance~$d_G(w,y)$ between $w$ and~$y$ is at least three, then every 2-simple path~$P$ between $w$ and~$y$ satisfies
    \begin{quote}
    \begin{itemize}
    \item[(d3)] The edge incident to~$y$ on~$P$ is a local minimum with value in $\lt[\nfrac23,\nfrac34\rt[$; and there is a second imin of~$P$ between this edge and
      $w$.
    \end{itemize}
    \end{quote}
  \end{enumerate}
\end{definition}


Before we prove that gluable graphs can be glued to each other, we review the examples from the beginning of this section.

\begin{example}
  The typed graphs with the decent labelings and root-vertices~$y$ described in the examples in the previous subsection are all gluable.  Checking this amounts to
  mechanically going through all the $t$-simple paths of the graphs.  We omit it here.
\end{example}


Let us now prove that gluing really works.

\begin{lemma}\label{lem:glue:1-sum}
  Let $(G_1,\tau_1,\phi_1,y)$ and $(G_2,\tau_2,\phi_2,y)$ be gluable, and let~$G$ result from gluing $G_1$ and~$G_2$ along~$\{y\}$.  Moreover, for all $e\in E$, let
  $\phi(e) := \phi_1(e)$, if $e\in E(G_1)$ and $\phi(e) := \phi_2(e)$, otherwise.
  Then $(G,\tau,\phi,y)$ is gluable.
\end{lemma}
The proof is purely mechanical and can be found in the appendix.


\begin{lemma}\label{lem:glue:2-sum}
  Let $(G_1,\tau_1,\phi_1,y_1)$ and $(G_2,\tau_2,\phi_2,y_2)$ be gluable, $w_1$ a connectable type-2 vertex of~$G_1$ and $w_2$ a connectable type-2 vertex of~$G_2$.  Let $G$ be the typed
  graph resulting from identifying $y_1$ with~$y_2$ to~$y$ and $w_1$ with~$w_2$.
  If $G_1$ and $G_2$ are gluable, and $d_{G_1}(y_1,w_1) + d_{G_1}(y_2,w_2) \ge 5$, then $(G,\tau,\phi,y)$ is gluable.
\end{lemma}

The condition on the distances, which means that identifying $y_1=y_2$ and $w_1=w_2$ cannot create a $C_4$ in~$G$, is needed because if the labels on two paths
satisfy the condition (d2.ii), then gluing them does not give a good edge-labeling.
The proof of Lemma~\ref{lem:glue:2-sum} can be found in the appendix.


The operation which adds a graph of the kind described in Example~\ref{exp:gluing:alphabetaCycle} differs from the above two.

Let~$(G_1,\tau_1,\phi_1,y_1)$ be a gluable graph, and $(y_1,u_1,v_1)$ a path in~$G$ as in Definition~\ref{def:gluable}(a).  Let~$H$ be a typed graph with
types~$\tau$, as described in Example~\ref{exp:gluing:alphabetaCycle}.  To specify the edge labeling of~$H$, we let $\alpha := \phi_1(y_1u_1)$ and $\beta :=
\phi_1(u_1v_1)$.  By Definition~\ref{def:gluable}(a), these values satisfy the conditions in Example~\ref{exp:gluing:alphabetaCycle} to define the decent
edge-labeling~$\phi_2$ of~$H$.
The proof of the following lemma is in the appendix.

\begin{lemma}\label{lem:glue:edge}
  The typed graph~$G'$ resulting from gluing $G$ and~$H$ along $\{y=y_1,u=u_1,v=v_1\}$ is gluable.
\end{lemma}

This is the only lemma that can create flags that are locked type-2 vertices.
The following lemma will give us the option to add sails that connect to these flags, the proof is analogous to that of Lemma~\ref{lem:glue:edge}.

Let~$(G,\tau_1,\phi_1,y_1)$ be a gluable graph, and $(y_1,u_1,w_1)$ a locking path in~$G$.  Let~$H$ be a typed graph with
types~$\tau$, as described in Example~\ref{exp:gluing:alphaCycle}.  To specify the edge labeling of~$H$, we let $\alpha := \phi_1(y_1u_1)$ and $\gamma :=
\phi_1(u_1w_1)$.  By Definition~\ref{def:lock}, these values satisfy the conditions in Example~\ref{exp:gluing:alphaCycle} to define the decent
edge labeling~$\phi_2$ of~$H$.

\begin{lemma}\label{lem:glue:edges}
  The typed graph~$G'$ resulting from gluing $G$ and~$H$ along $\{y=y_1,u=u_1,w=w_1\}$ is gluable.
\end{lemma}



\section{Non-existence of windmills}\label{sec:no-windmill}

In this section, we prove the following theorem mentioned in the introduction.

Again, in this section, $G$ is a critical graph of girth at least five.  Let~$W$ be a windmill in~$G$ with axis~$y$ and~$k$ sails, and denote by $\widebar W$ be the
subgraph of~$G$ induced by $W$ and all of its flags, regular or not.  We say that $\widebar W$ is the \textit{closure} of~$W$.  Define types $\bar\tau$ for~$\widebar
W$ as follows:
\begin{equation}\label{eq:def-tau-make-swell-subgraph}\tag{$*$}
  \bar\tau(v) =
  \begin{cases}
    2,                               &\text{if $v$ is a flag,}\\
    \deg_G(v) - \deg_{\widebar W}(v) &\text{otherwise.}
  \end{cases}
\end{equation}
We will prove the following.

\begin{lemma}\label{lem:nowindmill:main}
  The typed graph~$\widebar W$ with types $\bar\tau$ is decent, unless
  \begin{itemize}
  \item it contains an evil wheel, and there is no irregular flag;
  \item it contains an almost evil wheel, and there is an irregular flag.
  \end{itemize}
\end{lemma}

We will prove this lemma below.  Disregarding the types, from this lemma, we can immediately derive the following main result.

\begin{theorem}\label{thm:no-windmill}
  For every windmill~$W$ in~$G$, the closure $\widebar W$ of~$W$ contains an induced subgraph as depicted in Fig.~\ref{fig:gluing:almost-evil-wheels}, i.e., an
  (almost or not) evil wheel.
\end{theorem}
\begin{proof}
  This follows from Lemma~\ref{lem:nowindmill:main} by noting that $\widebar W$ is a swell subgraph of~$G$, and invoking Lemma~\ref{lem:fundamental}.
\end{proof}

The proof of Lemma~\ref{lem:nowindmill:main} is performed in two steps.  We first prove that the ``regular'' part of~$H$ is gluable, and then add the irregular flag,
if existent.
For this, let~$H$ be the subgraph of~$G$ induced by $W$ and its regular flags, and define types $\tau$ for~$H$ as in~\eqref{eq:def-tau-make-swell-subgraph}.
We now prove the following.

\begin{lemma}\label{lem:nowindmill:regular-gluable}
  There exists an edge-labeling $\phi$ such that $(H,\tau,\phi,y)$ is gluable, unless it contains an evil wheel.
\end{lemma}
\begin{proof}
  Suppose that~$H$ does not contain an evil wheel with axis~$y$.  Recall that this implies that~$H$ does not contain an evil wheel as a subgraph.
  
  We proove that~$H$ is gluable.  To do this, we associate to each of the basic elements (as laid down in Lemma~\ref{lem:windmills:flag-graph-inductive}) a gluable
  graph (one of the examples of the previous section); and to each of the construction rules, we associate one of the operations of Lemmas
  \ref{lem:glue:1-sum}--\ref{lem:glue:edge}.  By induction, this implies that~$H$ is gluable.

  \begin{enumerate}[$\bullet$]
  \item\textbf{Basic element~S} This corresponds to a typed path as in Example~\ref{exp:gluing:path-w-flag}.
  \item\textbf{Basic element~S$_-$} This corresponds to a cycle as in Example~\ref{exp:gluing:degsail}.
  \item\textbf{Basic element~S$^+$} This corresponds to a typed path as in Example~\ref{exp:gluing:path-wo-flag}.
  \item\textbf{Basic element~C2} This corresponds to a cycle in~$H$ as in Example~\ref{exp:gluing:fish}.  
  \item\textbf{Basic element~C4} This either corresponds to an almost evil wheel in~$H$, as in Example~\ref{exp:gluing:almost-evil-wheels}, or to a benign wheel, as
    in Example~\ref{exp:gluing:wheels}, because evil wheels are excluded.
  \end{enumerate}
  
  Suppose that the graph~$H'$ represented by a partial flag graph $F'$ is gluable.  We perform one of the construction rules to obtain an extended new flag
  graph~$F$, and explain how we use the gluing lemmas to extend~$H'$ to a gluable graph~$H$.
  
  \begin{enumerate}[$\bullet$]
  \item\textbf{Construction rule~U} This corresponds to taking a 1-sum as in Lemma~\ref{lem:glue:1-sum}.  The identification takes place at the axes of the
    components.
  \item\textbf{Construction rule~A} This corresponds to adding a path as in Example~\ref{exp:gluing:path-w-flag} or Example~\ref{exp:gluing:alphaCycle} via the 2-sum
    operation of Lemma~\ref{lem:glue:2-sum} or Lemma~\ref{lem:glue:edges}, respectively.  In the first case, the new sail-vertex from which the arc initiates
    corresponds to the path; the old flag-vertex which is the target of the arc identifies a flag~$w_1$ of~$H'$.  This flag~$w_1$ is identified with the vertex~$w$
    of the path.  The axis~$y_1$ is identified with the root vertex~$y_2$ of the path.  In the second case, the flag is the type-2 vertex $w$ in the bottom part of
    Fig.~\ref{fig:gluing:fishAndCycle}, with the bottom path connecting $y$ and~$w$ corresponding to the new sail-vertex.
  \item\textbf{Construction rule~B} This corresponds to adding a cycle as in Example~\ref{exp:gluing:alphabetaCycle} via Lemma~\ref{lem:glue:edge}.  The sail-vertex
    of~$F$ to which the new vertices are attached, identifies a sail (degenerate or not) in~$H$.  The two edges in this sail (or, on one path of the sail in the case
    when it is degenerate) which are closest to the axis~$y$ correspond to the two vertically drawn edges in the middle part of Fig.~\ref{fig:gluing:fishAndCycle},
    $yu$, $uv$.  The new flag-vertex is the type-2 vertex in that picture, and the new sail-vertex corresponds to the path between the root~$y$ and the type-0 vertex
    to the right of the type-2 vertex (the path which does not use the vertex~$u$).
  \end{enumerate}
\end{proof}

We point out the following property of the edge-labeling~$\phi$ constructed in this proof.
\begin{remark}\label{rem:property-of-sails-for-irreg}
  If~$W$ has an irregular flag~$w_0$, then on every sail~$P$ whose tip is adjacent to~$w_0$, the edge-labeling~$\phi$ for~$H$ has the labels shown in
  Fig.~\ref{fig:gluing:path-wo-flag}.
\end{remark}

In the second step, if necessary, we will need to add the irregular flag to~$H$.  This step will complete the proof of Lemma~\ref{lem:nowindmill:main}.

\begin{proof}[Proof of Lemma~\ref{lem:nowindmill:main}]
  If no irregular flag exists, this lemma is just a weaker form of Lemma~\ref{lem:nowindmill:regular-gluable}.  Suppose that an irregular flag in~$\widebar W$
  exists; denote it by~$w_0$.  We take the labeling~$\phi$ from Lemma~\ref{lem:nowindmill:regular-gluable}, and extend it to a decent labeling $\bar\phi$
  of~$\widebar W$.  To do this, let $y,x_1,\dots,x_r$ be the neighbors of~$w_0$ in~$\widebar W$.  We let $\bar\phi(e) := \phi(e)$ for all $e\in E(H)$; $\bar\phi(y
  w_0) := -10$; and $\bar\phi(w_0 x_j) := +1$, $j=1,\dots,r$.
  
  We now verify that the resulting labeling is decent, using the above Remark~\ref{rem:property-of-sails-for-irreg}.  Since there exists an irregular flag,
  by~\eqref{eq:def-tau-make-swell-subgraph}, we have $\tau(y) = 0$.  Since, in Definition~\ref{def:typed-graph-decent}, we only need to check 2-simple paths, the
  only paths we need to check are those starting or ending in~$w_0$.  Consider first paths staring in a type-2 vertex of~$H$ and ending in~$w_0$.  Such a path enters
  the sails whose tips are incident to~$w_0$ either through the axis or through a vertex adjacent to the axis.  In each of the two cases, the path touches a type-1
  vertex before it reaches~$w_0$.  By condition (b.1) of Definition~\ref{def:typed-graph-decent}, and using the fact that the edge with label~$-10$ on the path is an
  imin, we find that such a path has at least two imins, i.e., it satisfies condition (a.2).

  Secondly, consider a path~$P$ starting in~$w_0$ and ending in a type-1 vertex.  Since neither the axis nor the vertices adjacent to~$w_0$ are type-1 vertices, the edge with label~$-10$ is an imin not incident to the type-1 vertex in which~$P$ ends, and thus (b.1) is satisfied.

  We leave it to the reader to verify that~$\bar\phi$ is in fact a good edge-labeling of~$\widebar W$.
\end{proof}


\section{Proof of Theorem~\ref{thm:main}}\label{sec:proof-discharge}

Let $G$ be a minimum counter~example to Theorem~\ref{thm:main}, i.e., $G$ is a critical graph of girth at least~$5$ and with average degree less than~$3$.  Let
$\deg(v)$ denote the degree of vertex $v$.  To every vertex~$v$ assign a charge of $6-2\deg(v)$.  The total charge of the graph is $\sum_v (6-2\deg(v)) >0$, because
the average degree of~$G$ is less than~$3$.  Note that after the assignment of initial charges, only 2-vertices have positive charge.

Now we discharge the graph according to the following rule:
\begin{itemize}
\item 
  For every 2-vertex $u$ and every neighbour $v$ of $u$, if there are $k$ sails with tip $uv$, then $u$ sends $\frac{1}{k}$ charge (via these sails) to each
  $4^+$-vertex at the end of these $k$ sails.
\end{itemize} 

It can be seen that charges are sent from 2-vertices to $4^+$-vertices via paths consisting of only 3-vertices.  Now we show that after the discharging phase, every
vertex of the graph has nonpositive charge, a contradiction.  Indeed, let~$v$ be a vertex.  We consider the following cases.

\begin{enumerate}[(i)]
\item $v$ is a 2-vertex.  Then it has an initial charge of~2.  In the discharging, $v$ sends a total of 1 unit of charge out via each of the two tips, and $v$ does
  not receive any charge in the discharging phase.  So $v$ has 0 charge after the graph is discharged.
\item $v$ is a 3-vertex.  Then $v$ has an initial charge of 0.  Moreover, $v$ does not gain or lose any charge in the discharging phase.  So $v$ has 0 charge after
  the graph is discharged.
\item $v$ is a 4-vertex.  Then $v$ has an initial charge of $-2$.  To become positive, it must receive charges via at least three incident edges, implying that~$v$
  is the axis of a windmill.  (We note that this holds true even if two sails share a common tip and both end in~$v$.)  By Lemma~\ref{lem:nowindmill:main}, such a
  windmill must contain an evil or almost evil wheel as shown in Fig.~\ref{fig:gluing:evil-wheels} and Fig.~\ref{fig:gluing:almost-evil-wheels}.  It can be seen that
  in both cases, the axis of the windmill is at the same distance from the 2-vertices of the windmill as one of the flags.  Hence, vertex~$v$ receives at most
  $\frac{1}{2}$ via each sail of the wheel.  Thus, the charge of~$v$ after discharging is either at most $-2 + 3\cdot \frac{1}{2} = -\nfrac12$ or at most $-2 +
  2\cdot \nfrac12 + 1 = 0$.
\item $v$ is a 5-vertex.  Then $v$ has an initial charge of $-4$.  To become positive, it must receive charges via every incident edge, implying that~$v$ is the axis
  of a 5-windmill in~$G$.  Again, similar to the above case, Lemma~\ref{lem:nowindmill:main} implies that such a windmill contains an evil or almost evil wheel in
  both of which cases, the axis of the windmill is of the same distance from the 2-vertices of the windmill as one of the flags.  Hence, vertex~$v$ receives at most
  $\frac{1}{2}$ via each sail of the wheel.  Thus, the charge of~$v$ after discharging is either at most $-4 + 4\cdot 1 = 0$ or at most $-4 + k\cdot \frac{1}{2} +
  (5-k) \times 1\geq 0$ where $k\geq 2$.
  \item $v$ is a $6^+$-vertex. Then $v$ has an initial charge of $6-2 \deg(v)$.  Since $v$ receives at most 1 unit of charge via each incident edge, $v$ has at most
    $(6-2 \deg(v)) + deg(v) \leq 0$ charge after the graph is discharged.
\end{enumerate}



\section{Conclusion}

We have seen that imposing a lower bound on the girth facilitates the construction of good edge-labelings, or even decent edge-labelings.  In this paper, we
have used this approach together with a degree-bound.  It seems probable that high girth benefits other open problems about good edge-labeling.  For example,
Ara\'ujo et al.~\cite{AraujoCohenGiroireHavet11} conjecture that for every $c<4$, the number of (pairwise non-isomorphic) critical graphs with average degree at
most $c$ is finite.  We propose the following weakening of their conjecture.

\begin{conjecture}\label{conj:avg-deg-c}
  For every $c<4$, the number of (pairwise non-isomorphic) critical graphs with girth at least five and average degree at most $c$ is finite.
\end{conjecture}

This paper settles the case $c=3$.  For $c=3$, but without restriction to the girth, a modification of Conjecture~\ref{conj:main} proposes itself naturally:

\begin{conjecture}[Ara\'ujo-Cohen-Giroire-Havet/modified]
  There is no critical graph with average degree less than~3, with the exception of~$C_3$, $K_{2,3}$, and the graph displayed in
  Fig.~\ref{fig:girth-4-counterexp}.
\end{conjecture}

\subsection*{Acknowledgments}

We would like to thank the anonymous referees for their careful work.


\providecommand{\bysame}{\leavevmode\hbox to3em{\hrulefill}\thinspace}
\providecommand{\MR}{\relax\ifhmode\unskip\space\fi MR }
\providecommand{\MRhref}[2]{%
  \href{http://www.ams.org/mathscinet-getitem?mr=#1}{#2}
}
\providecommand{\href}[2]{#2}


\appendix
\section{Deferred proofs}\label{apx:glue}

\begin{proof}[Proof of Lemma~\ref{lem:glue:1-sum}]
  The conditions of Definition~\ref{def:gluable} are satisfied, since they require to check paths originating from the root vertex~$y$ only.
  Moreover, $\phi$ is a good edge-labeling.
  It remains to show that~$\phi$ is decent.
  
  To verify property~(a) of Definition~\ref{def:swell-graph-decent}, let~$Q$ be a 2-simple path between two \mbox{type-2} vertices~$w_1 \in V(G_1)\setminus V(G_2)$
  and~$w_2 \in V(G_2)\setminus V(G_1)$.  Let~$P_1 := w_1Qy$ and $P_2 := yQw_2$.
  If both $Q_1$ and~$Q_2$ satisfy (d2.i) then the edges incident to~$y$ in neither $Q_1$ nor $Q_2$ are part of the respective imins, and (a.2) of
  Definition~\ref{def:swell-graph-decent} holds.
  If $Q_1$ satisfies (d2.i) and $Q_2$ satisfies (d2.ii) then (a.2) of Definition~\ref{def:swell-graph-decent} holds: one of the imins is the imin of~$Q_1$, the other
  is the edge of $Q_2$ incident on~$y$.
  If both $Q_1$ and~$Q_2$ satisfy (d2.ii) then (a.1) holds, the imin there being the path of length two consisting of the two edges of~$Q$ incident on~$y$.
  If $Q_1$ satisfies (d2.i) and $Q_2$ satisfies (d3), then (a.2) holds for~$P$.
  If $Q_1$ satisfies (d2.ii) and $Q_2$ satisfies (d3), then (a.2) holds for~$P$.
  If both $Q_1$ and~$Q_2$ satisfy (d3), then (a.2) holds for~$P$.
  
  To verify property~(b), let~$Q$ be a 1-simple path between a type-1 vertex $w_1 \in V(G_1)\setminus V(G_2)$ and a type-2 vertex~$w_2 \in V(G_2)\setminus V(G_1)$.
  Note that the property in~(b) of Definition~\ref{def:gluable} holds for~$Q_1$.
  If (d2.i) holds for~$Q_2$, then the imin of~$Q_2$ is an imin of~$P$ not incident on~$w_1$.
  If (d2.ii) holds for~$Q_2$, then the edge incident on~$y$ in~$Q_2$ is an imin of~$P$ not incident on~$w_1$.
  If (d3) holds for~$Q_2$, then the imin of~$Q_2$ closer to~$w_2$ is an imin of~$P$ not incident on~$w_1$.
\end{proof}

\begin{proof}[Proof of Lemma~\ref{lem:glue:2-sum}]
  Denote the vertex of~$G$ resulting from identifying $y_1$ and~$y_2$ by~$y$, and the one resulting from identifying $w_1$ and~$w_2$ by~$w$.

  Let us first check Definition~\ref{def:gluable}(a--d).
  Property Definition~\ref{def:gluable}(a) is satisfied because no new path of this kind is added.  The conditions of Definition~\ref{def:gluable}(b--d) are
  satisfied, since they require to check 1- and 2-simple paths originating from the root vertex~$y$ only: these paths cannot contain~$w$, and are thus contained
  entirely in either $G_1$ or~$G_2$.

  We have to make sure that~$\phi$ is good, and that it satisfies the conditions (a) and~(b) of Definition~\ref{def:swell-graph-decent}.  We may assume w.l.o.g.\
  that $d_{G_2}(y_2,w_2) \ge 3$.

  We first prove that~$\phi$ is good.  For this, let~$Q_1$ be a path in~$G_1$ between $y$ and~$w_1$, and let~$Q_2$ be a path in~$G_2$ between $y$ and~$w_2$.  We have
  to verify that the cycle $C := Q_1 + Q_2$ has two local minima.
  
  If (d2.i) holds for~$Q_1$ and (d3) for~$Q_2$, then $C$ has two local minima.
  If (d2.ii) holds for~$Q_1$ and (d3) for~$Q_2$, then $C$ has two local minima.
  If both $Q_1$ and~$Q_2$ satisfy (d3), then $C$ has two local minima.
  
  Secondly, we prove properties (a) and~(b) of Definition~\ref{def:swell-graph-decent} hold.
  Note that for both these properties, we do not need to consider paths containing~$w$ as an interior vertex, because those are not 2-simple (in the case of (a)) or
  even 1-simple (for (b)).
  But this leaves us with the same situation which we have checked in the previous lemma.
\end{proof}

\begin{proof}[Proof of Lemma~\ref{lem:glue:edge}]
  We start by proving that~$\phi$ is a good edge-labeling.  For this, let~$C$ be a cycle in~$G$ containing edges of both $E(G_1)\setminus E(H)$ and $E(H)\setminus
  E(G_1)$.  Such a cycle can be turned into a cycle $C'$ in~$G_1$ by replacint the path of~$C$ in $E(H)\setminus E(G_1)$ by the single edge~$u_1y_1$.  We show that
  the fact that there are two local minima $Q_1$, $Q_2$, on~$C'$ implies that there are two local minima on~$C$.
  
  Obviously, if any of the local minima of~$C'$ contains neither $u_1$ nor~$y_1$, then it is a local minimum of~$C$.
  On the other extreme, if one of the two, say~$Q_2$, contains the edge~$u_1y_1$, then $Q_1$ and the path $P_{-1}$ formed by the two edges with label $-1$
  in~$C\setminus E(G_1)$ are two distinct local minima, because $-1 < \alpha$.
  Thus, we have to make sure that if any of the local minima of~$C'$ contains exactly one of the vertices $u_1$ or~$y_1$, then it can be modified to be a local
  minimum of~$C$.  Firstly, suppose $Q_1$ has value $\mu_1$ and contains~$u_1$ but not~$y_1$.  If $\mu_1 < -1$, then $Q_1$ is a local minimum of~$C$; if~$\mu_1 > -1$
  then $P_{-1}$ is a local minimum of~$C$; if the two are equal, then $Q_1+P_{-1}$ is a local minimum of~$C$.
  Secondly, suppose $Q_2$ has value $\mu_2$ and contains~$y_1$ but not~$u_1$.  If $\mu_2 < \frac{\nfrac23+\alpha}{2}$, then~$Q_2$ is a local minimum of~$C$; if
  $\mu_2 > \frac{\nfrac23+\alpha}{2}$, then the path $\dot P$ formed by the edge of $C\setminus E(G_1)$ with label $\frac{\nfrac23+\alpha}{2}$ is a local minimum
  of~$C$; if the two are equal, then $Q_2 + \dot P$ is a local minimum of~$C$.
  
  Next, we have to show that the edge-labeling~$\phi$ satisfies the properties (a) and~(b) of Definition~\ref{def:typed-graph-decent}.
  For propery~(a), let $w$ be the type-2 vertex of~$H$, let~$w_1$ be any type-2 vertex of~$G_1$, and let~$P$ be a $w$-$w_1$-path in~$G$.  On the one hand, if~$u=u_1$
  is on~$P$, then by Definition~\ref{def:typed-graph-decent}(b.1) applied to~$P(w_1,u_1)$, $Q$ has one imin not incident on~$u$, and the edge of~$Q$ incident to~$w$
  is a second, distinct, imin.  On the other hand, if~$y=y_1$ is on~$P$, we use (d2.i), (d2.ii), or~(d3) of Definition~\ref{def:gluable} for the path
  $Q':=Q(y_1,w_1)$.  Indeed, 
  if~$Q'$ satisfies~(d2.i) the length of $Q(w,y)$ is two (i.e., it contains the edge~$uy$),  then~$Q$ has two imins;
  if~$Q'$ satisfies~(d2.i) the length of $Q(w,y)$ is at least three, then~$Q$ has two imins;
  if~$Q'$ satisfies~(d2.ii) the length of $Q(w,y)$ is two, then~$Q$ has two imins;
  if~$Q'$ satisfies~(d2.ii) the length of $Q(w,y)$ is three, then~$Q$ has two imins;
  if~$Q'$ satisfies~(d3) the length of $Q(w,y)$ is two, then ~$Q$ has two imins;
  if~$Q'$ satisfies~(d3) the length of $Q(w,y)$ is at least three, then ~$Q$ has two imins.

  Finally, to check the conditions of Definition~\ref{def:gluable}, the only kind of paths which are added beyond those which were present in $G_1$ and~$H$ are those
  which result from taking a path $Q_1$ in~$G_1$ from~$y_1$ to $u_1=u$, and adding the edge $uw$ of~$H$.  Invoking the condition~(b.1) of
  Definition~\ref{def:swell-graph-decent}, $Q_1+uw$ contains two imins: one on~$Q_1$ and the other being the edge~$uw$.
\end{proof}


\end{document}